\documentclass[11pt]{amsart}
\usepackage[colorlinks,citecolor=red, dvipdfm]{hyperref}
\usepackage{extarrows}
\usepackage[all]{xy}

\newtheorem{theorem}{Theorem}[section]
\newtheorem{lemma}[theorem]{Lemma}

\theoremstyle{definition}

\newtheorem{example}[theorem]{Example}

\newtheorem{proposition}[theorem]{Proposition}

\newtheorem{remark}[theorem]{Remark}

\newtheorem{notation}[theorem]{Notation convention}
\theoremstyle{remark}

\newcommand{\be}{\begin{equation}}
\newcommand{\ee}{\end{equation}}

\numberwithin{equation}{section}

\begin{document}
\title[Combinatorial identities and Chern numbers]{Combinatorial identities and Chern numbers of complex flag manifolds}
\author{Ping Li}
\address{School of Mathematical Sciences, Tongji University, Shanghai 200092, China}
\email{pingli@tongji.edu.cn\\
pinglimath@gmail.com}
\author{Wenjing Zhao}
\thanks{Both authors were partially supported by the National
Natural Science Foundation of China (Grant No. 11471247) and the
Fundamental Research Funds for the Central Universities.}




\begin{abstract}
We present in this article a family of new combinatorial identities via purely differential/complex geometry methods, which include as a speical case a unified and explicit formula for Chern numbers of all complex flag manifolds. Our strategy is to construct concrete circle actions with isolated fixed points on these manifolds and explicitly determine their weights. Then applying Bott's residue formula to these models yields the desired results.
\end{abstract}

\maketitle
\section{Introduction}\label{section1}
Complex flag manifolds are natural generalizations of complex projective spaces and complex Grassmannian manifolds, and forms an important subclass of compact complex manifolds, which can be described as follows. Arbitrarily fix a positive integer $r$ and $r+1$ positive integers $m_1,m_2,\ldots,m_r,m_{r+1}$ and set $N:=\sum_{j=1}^{r+1}m_j.$
Define the following set
\be\label{definitionF}
\begin{split} &F:=F(m_1,\ldots,m_r,m_{r+1}):=\\
&\bigg\{(L_1,\cdots,L_r)~\bigg|~L_1\subset
L_2\subset\cdots\subset L_r, \text{$L_i$ are linear subspaces in $\mathbb{C}^N$}, \text{dim}_{\mathbb{C}}L_i=\sum_{j=1}^im_j\bigg\}
\end{split}\ee
and call such an $(L_1,\cdots,L_r)\in F$ a \emph{flag} in $\mathbb{C}^N$. When $r=1$ or $r=1$ and $m_1=1$, it degenerates to the complex Grassmannian consisting of complex $m_1$-dimensional linear subspaces in $\mathbb{C}^{m_1+m_2}$ or complex $m_2$-dimensional projective spaces, which play fundamental roles in geometry and topology.

Note that the unitary group $\text{U}(N)$ acts transitively in a natural manner on $F$ and its isotropy subgroup is $\text{U}(m_1)\times\cdots\times\text{U}(m_{r+1})$. So the complex flag manifold $F$ can be endowed with a homogeneous space structure
\be\label{G/U}\frac{\text{U}(N)}{\text{U}(m_1)
\times\cdots\times\text{U}(m_{r+1})}.\ee

It is a classical fact, essentially due to Borel, Koszul, Wang and Matsushima that this homogenous space admits a canonical $\text{U}(N)$-invariant complex structure and can be endowed on this canonical complex structure with a unique invariant K\"{a}hler-Einstein metric with positive scalar curvature up to rescaling (\cite{Bor}, \cite{Ko}, \cite{Wa}, \cite{Mat}). Borel and Hirzebruch (\cite{BH1}, \cite{BH2}) systematically investigated the characteristic classes of homogeneous space $G/U$ ($G$ is a compact connected Lie group and $U$ is its closed subgroup) in terms of Lie theoretical information of $G$ and $U$ (root systems, weights, representations and so on). In particular, they gave a complete characterization in terms of root systems of when an invariant almost-complex structure on the complex flag manifolds $F$ is integrable and of the number of inequivalent invariant structures on $F$ (\cite[\S 12-\S 14]{BH1}).

As is well-known Chern numbers are basic numerical invariants of compact (almost) complex manifolds, which are complete invariants for complex cobordism (\cite{MS}). So a natural question is to calculate/determine Chern numbers of the complex flag manifold $F=F(m_1,\ldots,m_r,m_{r+1})$ endowed with possibly various almost-complex structures. A direct way to approach this question is to apply the above-mentioned Borel-Hirzebruch theory in \cite{BH1}, where the Chern classes of $F$ are described as polynomials in the roots of unitary groups. Indeed this idea has been taken up by Kotschick and Terzi\'{c} in \cite{KT} where they calculated the Chern numbers of the two invariant complex structures of complex flag manifold $F(n,1,1)$ ($r=2$ and $m_2=m_3=1$) and gave some related applications to the geometry of $F(n,1,1)$.

Besides this direct method, there is another indirect way to attack this problem, which is what we shall adopt. Note that usually it is difficult to calculate Chern numbers directly from their definition. A remarkable result of R. Bott (\cite{Bo}), which is now called Bott's residue formula,
tells us that if this compact complex manifold has a holomorphic vector field whose zero point set is isolated and admits a non-degenerate condition, we can reduce the calculation of its Chern numbers to the
consideration of local information around the zero point set of this vector field. This non-degenerate condition is automatically satisfied when the vector fields are generated by circle actions (see Section \ref{Bottresiduesection} for more details). When the zero point set is not isolated, a similar result was
established by Atiyah and Singer in \cite[$\S$8]{AS}, which is a
beautiful application of their general Lefschetz fixed point formula
and is now commonly called the Atiyah-Bott-Singer residue formula.

\emph{The main purpose} of this paper is to apply Bott's residue formula to \emph{all} complex flag manifolds $F(m_1,\ldots,m_{r+1})$ to obtain a family of new combinatorial identities, Theorem \ref{mainresult}, which \emph{strictly} include a unified and explicit formula for calculating Chern numbers of complex flag manifolds as a special case.
 The reason why our main results Theorem \ref{mainresult} are stronger than just a formula for Chern numbers is due to the statement of Bott's residue formula itself. If we take a closer look at the precise statement of Bott's residue formula, we shall see that
it provides \emph{more vanishing-type
information} for low-degree polynomials
than just a method of calculating Chern numbers
(more details can be found in
Section \ref{Bottresiduesection}). Thus accordingly our main results contain more information.

Note that our inputs (complex flag manifolds and Bott's residue formula) are purely geometric. So it is a little surprising to see that the outputs (Theorem \ref{mainresult}) are essentially combinatorial identities. However, it has been widely known that various aspects of geometry and topology of complex flag manifolds are deeply related to enumerative combinatorics (\cite{Ma}, \cite{Fu}). So our main results in this paper strengthen this point of view and thus in this sense they are natural and expectable.

\subsection*{Outline of this paper}
The rest of this paper is organized as follows. we shall state our main results and present two examples in Section \ref{mainresultsection}. In Section \ref{Bottresiduesection} we briefly review Bott's residue formula in the case of holomorphic circle action with isolated fixed points. Section \ref{preliminarysection} is devoted to preliminaries on complex flag manifolds and explicit construction of their local coordinate charts.  After these preliminaries, we shall construct in Section \ref{constructionsection} concrete circle actions on complex flag manifolds with isolated fixed points and explicitly determine the weights around these fixed points with the help of the local coordinate charts established in Section \ref{preliminarysection}, from which the proof of our main result easily follows. The last section is an Appendix which contains a detailed calculation on a Chern number discussed in Example \ref{exmple2} in Section \ref{section1}.

\section*{Acknowledgements}
This paper was completed during the first author's visit to Max-Planck Institute for Mathematics at Bonn in Fall 2016, to whom the first author would like to express his sincere thanks for its hospitality and financial support.

\section{Main results}\label{mainresultsection}
We introduce in this section some necessary notation and symbols and then state our main results of this paper.

We arbitrarily fix as in Section \ref{section1} a positive integer $r$ and $r+1$ positive integers $m_1,\ldots,m_{r+1}$ and define $N:=\sum_{j=1}^{r+1}m_j$ and $F:=F(m_1,\ldots,m_{r+1})$ as in (\ref{definitionF}). An ordered sequence $I:=(I_1,\ldots,I_{r+1})$ is called a \emph{decomposition}
of the set $\{1,2,\ldots,N\}$ if
 \be\label{decomposition}I_i\subset\{1,2,\ldots,N\}, \qquad\bigcup_{i=1}^{r+1}
 I_i=\{1,2,\ldots,N\},\qquad\sharp(I_i)=m_i.\ee
Here $\sharp(\cdot)$ denotes the cardinality of a set.
This means that these $I_i$ are mutually disjoint. Note that
there are
\be\label{eulernumber}{N\choose m_1}{N-m_1\choose m_2}\cdots{m_r+m_{r+1}\choose m_r}=\frac{N!}{m_1!m_2!\cdots m_r!m_{r+1}!}\ee
different decompositions for this set $\{1,2,\ldots,N\}$, which is precisely the Euler characteristic of the complex flag manifold $F(m_1,\ldots,m_{r+1})$ (see Remark \ref{remarkaftertheorem}).

Let $x_1,\ldots,x_N$ be $N$ variables. For each decomposition $I=(I_1,\ldots,I_{r+1})$, we formulate a set
$W_I$ as follows.
$$W_I:=\bigcup_{1\leq i<j\leq r+1}\{-x_{\alpha}+x_{\beta}~|~\alpha\in I_i,\beta\in I_j\}.$$
Note that $$\sharp(W_I)=\sum_{1\leq i<j\leq r+1}m_im_j=:d,$$ which is exactly the complex dimension of the complex flag manifold $F$ \big(this fact can also be seen from (\ref{G/U}) as $\text{dim}_{\mathbb{R}}\text{U}(n)=n^2$\big).

\begin{example}\label{examplepermutation}
Suppose that $m_1=m_2=\cdots=m_{r+1}=1$ and then $N=r+1$. In this case the decompositions correspond to $S_N$, the permutation group on $N$ objects,
$$W_{\sigma}=\{-x_{\sigma(i)}+x_{\sigma(j)}~|~1\leq i<j\leq N\},\qquad\forall~\sigma\in S_N,$$
and $\sharp(W_{\sigma})=N!.$
\end{example}

Recall that a \emph{partition} $\lambda$ is a finite sequence of positive integers
$(\lambda_1,\lambda_2,\ldots,\lambda_l)$
 in non-increasing order:
$\lambda_1\geq\lambda_2\geq\cdots\geq\lambda_l\geq 1.$
The \emph{weight} of $\lambda$ is defined to be
$\sum_{i=1}^{l}\lambda_i.$

We denote by $c_i(\cdot,\ldots,\cdot)$ the $i$-th elementary symmetric polynomial of $d$ variables. If $\lambda=(\lambda_1,\lambda_2,\ldots,\lambda_l)$
is a partition, we define $$c_{\lambda}(\cdot,\ldots,\cdot):=\prod_{i=1}^{l}c_{\lambda_i}(\cdot,\ldots,\cdot)$$ to be the product of these elementary symmetric polynomials $c_{\lambda_i}$. It will be clear soon that the Chern numbers of $F$ corresponding to the partition $\lambda$ shall be described in terms of the quantity $c_{\lambda}$.

\begin{notation}
If $\{y_1,\ldots,y_n\}$ is a set of $n$ variables, we define
$$e(\{y_1,\ldots,y_n\}):=\prod_{i=1}^{n}y_i$$
 to be the product of the elements in the set.
If $f(\cdot,\ldots,\cdot)$ is a symmetric polynomial of $n$ variables, we define
$$f(\{y_1,\ldots,y_n\}):=f(y_1,\ldots,y_n).$$
\end{notation}

With the above-defined symbols and notation understood, we can now state our main results as follows.

\begin{theorem}[Main results]\label{mainresult}
Suppose that $f(\cdot,\ldots,\cdot)$ is a homogeneous symmetric polynomial of $d$ variables. We denote by $\text{deg}(f)$ the degree of $f$. Then we formulate a rational function of the variables $x_1,\ldots,x_N$ as follows.
\be R_f(x_1,\ldots,x_N):=\sum_{I=(I_1,\ldots,I_{r+1})}
\frac{f(W_I)}{e(W_I)},\label{expression}\ee
where the sum of the right hand side is over all the decompostions $I$ of the set $\{1,\ldots,N\}$. Then we have
\begin{enumerate}
\item
if $\text{deg}(f)<d$, then $R_f(x_1,\ldots,x_N)\equiv 0.$

\item
If $\text{deg}(f)=d$, then $R_f(x_1,\ldots,x_N)$ is a constant depending only on $f$. Moreover, if $\lambda$ is a partition of weight $d$, then $R_{c_{\lambda}}(x_1,\ldots,x_N)$ is the Chern number of the complex flag manifold $F(m_1,\ldots,m_r,m_{r+1})$ corresponding to the partition $\lambda$.
\end{enumerate}
\end{theorem}

\begin{remark}\label{remarkaftertheorem}~
\begin{enumerate}
\item
If $\text{deg}(f)<d$, it is quite difficult, at least at the first glance, to imagine that the right hand side of (\ref{expression}) vanishes. So this situation provides us a family of nontrivial combinatorial identities.

\item
If $\text{deg}(f)=d$, Theorem \ref{mainresult} tells us that $R_f(x_1,\ldots,x_N)$ is a constant depending only on $f$. This means that, for any $N$ mutually distinct numbers (integral, real or complex etc) $(\alpha_1,\ldots,\alpha_N)$, $R_f(\alpha_1,\ldots,\alpha_N)$ is equal to this constant. This provides us an effective method to calculate Chern numbers of the complex flag manifolds in practice.

\item
Put $f=c_d$ in (\ref{expression}), then $c_d(W_I)=e(W_I)$ and thus each summand in the right hand side of (\ref{expression}) is $1$ and so
$$R_{c_d}(x_1,\ldots,x_N)=\frac{N!}{m_1!m_2!\cdots m_r!m_{r+1}!},$$
 which is equal to the Chern number of the complex flag manifold $F(m_1,\ldots,m_{r+1})$ corresponding to the partition $(d)$. This Chern number is famously known to be equal to the Euler characteristic of $F(m_1,\ldots,m_{r+1})$ (compare to (\ref{eulernumber})).
\end{enumerate}
\end{remark}

If $\text{deg}(f)>d$, the right hand side of (\ref{expression}) is still well-defined. But in general the expressions $R_f(x_1,\ldots,x_N)$ depend on $x_1,\ldots,x_N$ and thus lack geometrical meanings. Nevertheless, in our situation, for two special homogenous symmetric polynomials of degree $d+1$: $c_1^{d+1}$ and $c_dc_1$, we still have the following
\begin{proposition}\label{mainprop}
\begin{eqnarray}\label{d}
\left\{ \begin{array}{ll}
R_{c_1^{d+1}}(x_1,\ldots,x_N)=0\\
R_{c_dc_1}(x_1,\ldots,x_N)=0
\end{array} \right.
\end{eqnarray}
\end{proposition}

\begin{remark}~
\begin{enumerate}
\item
The reason for the first equality in (\ref{d}) is related to the residue formula of the Futaki integral invariant (\cite{FM2}), which obstructs the existence of K\"{a}hler-Einstein metrics on Fano manifolds (compact complex manifolds with positive first Chern classes). The well-known existence of such a metric on $F(m_1,\ldots,m_{r+1})$ leads to the first equality in (\ref{d}). We shall explain this in more detail in Section \ref{Bottresiduesection}.

\item
In contrast to the nontriviality of the first one in (\ref{d}),  the second one of (\ref{d}) is quite obvious:
$$R_{c_dc_1}(x_1,\ldots,x_N)=\sum_Ic_1(W_I)=\sum_I(\text{the sum of elements in $W_I$})=0,$$
which is indeed a special case of a general result proved by the first author in \cite{Li4}, which we will mention again in Section \ref{Bottresiduesection}.
\end{enumerate}
\end{remark}

Before ending this section, we would like to illustrate Theorem \ref{mainresult} by two simple examples.

\begin{example}\label{example1}
As in Example \ref{examplepermutation} we assume $m_1=\cdots=m_{r+1}=1$ and $r+1=N$. Then we have
\be\begin{split}
&\sum_{\sigma\in\text{S}_N}
\frac{f\big(\{-x_{\sigma(i)}+x_{\sigma(j)}~|~1\leq i<j\leq n\}\big)}
{\prod_{1\leq i<j\leq N}(-x_{\sigma(i)}+x_{\sigma(j)})}
\\
=&\left\{ \begin{array}{ll}
0, & \text{deg}(f)<\frac{N(N-1)}{2},\\
\text{constant depending on $f$}, &
\text{deg}(f)=\frac{N(N-1)}{2}.
\end{array} \right.
\end{split}\nonumber
\ee
\end{example}

\begin{example}\label{exmple2}

Note that
$$F(1,1,2)\xlongequal{C^{\infty}}\frac{U(4)}{U(1)\times
U(1)\times U(2)}\xlongequal{C^{\infty}}\frac{U(4)}{U(1)\times
U(2)\times U(1)}\xlongequal{C^{\infty}}F(1,2,1).$$
Here ``$C^{\infty}$" means diffeomorphism. This means that as smooth manifolds $F(1,1,2)$ and $F(1,2,1)$ are diffeomorphic to each other. However, Borel and Hirzebruch applied the results obtained in \cite{BH1} to show that (\cite[\S 24.11]{BH2}) the Chern numbers $c_1^5$ of them are
\be\label{chernnumber}c_1^5[F(1,1,2)]=4500\neq4860=c_1^5[F(1,2,1)],\ee
which implies that the canonical complex structures on $F(1,1,2)$ and $F(1,2,1)$ are \emph{different}. This gives the first example of two compact complex manifolds which are diffeomorphic to each other, but have different Chern numbers (hence not biholomorphic to each other). We can also apply Theorem \ref{mainresult} to calculate them, whose details are presented in Appendix \ref{appendix}.
This was revisited again by Hirzebruch in \cite{Hi}.
\end{example}

\begin{remark}
It can be shown that (cf. \cite{Hi} or \cite{KT}) the complex flag manifolds $F(1,1,n)$ can be identified with $\mathbb{P}(T\mathbb{C}P^{n+1})$, the projectivizations of the holomorphic tangent of the complex projective space $\mathbb{C}P^{n+1}$. This point of view was taken up in \cite[\S 4]{KT} to deduce the Chern number $c_1^{2n+1}[F(1,1,n)]$ for general $n$. In principle, this formula can also be obtained directly via our Theorem \ref{mainresult}. However, practicely for general $n$ the expression we need to deal with is quite complicated and thus it is difficult to obtain the closed formula as in \cite[p. 604, Theorem 3]{KT} without resorting to the characteristic classes of $\mathbb{P}(T\mathbb{C}P^{n+1})$.
\end{remark}

\section{Bott's residue formula}\label{Bottresiduesection}
In this section we briefly review Bott's residue formula for circle actions and give some related remarks.

Suppose that $M$ is a compact complex manifold with complex dimension $n$ and it is equipped with a holomorphic circle action with isolated fixed points. We denote by $P$ such an isolated fixed point. At each $P$ there are well-defined $n$ integers
$k_{1},\cdots, k_{n}$ (not necessarily distinct) induced
from the isotropy representation of this circle action on the holomorphic tangent space $T_PM$. Namely, the circle acts on $T_PM\cong\mathbb{C}^n$ in the following manner:
\be\label{weightdefinitin}\begin{split}
S^1\times\mathbb{C}^n&\longrightarrow\mathbb{C}^n,\\
\big(g,(v_1,\ldots,v_n)\big)&\longmapsto(g^{k_1}\cdot v_1,\ldots,g^{k_n}\cdot v_n).\end{split}\ee
Note
that these $k_{1},\cdots,k_{n}$ are \emph{nonzero} as
these fixed points $P$ are isolated and called \emph{weights} at $P$ with respect to this circle action.

Following the notation and symbols in Section \ref{mainresultsection}, let $f(x_{1},\cdots,x_{n})$ be a homogeneous
symmetric polynomial in the variables $x_{1},\cdots,x_{n}$. Then
$f(x_{1},\cdots,x_{n})$ can be written in an essentially unique way
in terms of the elementary symmetric polynomials
$\widetilde{f}(c_{1},\cdots,c_{n})$, where
$c_{i}=c_{i}(x_{1},\cdots,x_{n})$ is the $i$-th elementary symmetric
polynomial of $x_{1},\cdots,x_{n}$. For instance,
$$f(x_{1},\cdots,x_{n})=x_1^2+\cdots x_n^2=c_1^2-2c_2=\widetilde{f}(c_1,\ldots,c_n).$$

With the above-mentioned notation and symbols understood,  we can now state a version of the Bott's residue formula \cite{Bo},
which reduces the calculation of Chern numbers of $M$ to
the weights $k_i$ around the isolated fixed points of this holomorphic circle action.
\begin{theorem}[Bott's residue formula]\label{BRF0}
Suppose a compact complex $n$-dimensional manifold $M$ admits a holomorphic circle action with isolated fixed points $\{P\}$ and the weights around each $P$ are denoted by $k_{1},\cdots,k_{n}$, which depend on $P$. Then we have
\begin{eqnarray}\label{BRF}
\sum_P\frac{f(k_{1},\ldots,k_{n})}{\prod_{i=1}^{n}k_i}
=\left\{ \begin{array}{ll}
0, & \text{deg}(f)<n,\\
\int_M\widetilde{f}(c_{1}(M),\cdots,c_{n}(M)), &
\text{deg}(f)=n.
\end{array} \right.
\end{eqnarray}
Here the sum is over all the isolated fixed points $P$ and $c_i(M)$ denotes the $i$-th Chern class of $M$.
\end{theorem}

\begin{remark}\label{remark}~
\begin{enumerate}
\item
If $\lambda$ is partition of weight $n$ and $\widetilde{f}(c_1,\ldots,c_n)=c_{\lambda}$, this formula precisely gives a method to calculate the Chern number of $M$ with respect to the partition $\lambda$ in terms of the weights $k_i$ around the isolated fixed pionts.

\item
The conclusion that the left hand side of (\ref{BRF}) vanishes when $\text{deg}(f)<n$ is by no means trivial. Thus this formula provides us
a family of vanishing-type results about the weights $k_i$ around the isolated fixed points rather than just how to
compute the Chern numbers of $M$. This observation played a dominant role in establishing main results in the first author's previous works (\cite{Li1}, \cite{Li1.5}, \cite{Li2}, \cite{LL1} and \cite{LL2}).
\end{enumerate}
\end{remark}

Note that (\ref{BRF}) tells us \emph{nothing} about those $f$ with $\text{deg}(f)>n$. For general $f$ with $\text{deg}(f)>n$ the left hand side of (\ref{BRF}) may not vanish or may depend on the weights $k_i$, which can be easily tested by the data presented in Appendix \ref{appendix}. Nevertheless, for two specail $f$ of degree $n+1$: $c_1^{n+1}$ and $c_nc_1$, we have the following

\begin{theorem}\label{BRF2}
We make the same assumptions as in Theorem \ref{BRF0}.
\begin{enumerate}
\item
If $M$ admits a K\"{a}hler-Einstein metric, then
\be\label{Futaki}
\sum_P\frac{c_1^{n+1}(k_{1},\ldots,k_{n})}{\prod_{i=1}^{n}k_i}
=\sum_P\frac{(k_{1}+\cdots+k_{n})^{n+1}}{\prod_{i=1}^{n}k_i}=0\ee

\item
The following identity holds {\rm(}unconditionally{\rm)}
\be\label{li}\sum_P\frac{c_nc_1(k_{1},\ldots,k_{n})}{\prod_{i=1}^{n}k_i}
=\sum_P(k_{1}+\cdots+k_{n})=0.\ee
\end{enumerate}
\end{theorem}

\begin{remark}~
\begin{enumerate}
\item
The left hand side of (\ref{Futaki}) was showed by Futaki-Morita (\cite[Prop. 2.3]{FM2}) to be the Futaki integral invariant with respect to the holomorphic vector field generated by the circle action, which vanishes if $M$ admits a K\"{a}hler-Einstein metric (\cite{Fut}). Some of the considerations in \cite{FM1} and \cite{FM2} have recently been improved by the first author in \cite{Li1.5}.

\item
(\ref{li}) is a particular case of \cite[Corollary 1.3]{Li4} by the first author, which is in turn an application of the rigidity phenomenon of Dolbeault-type operators on compact (almost) complex manifolds.
\end{enumerate}
\end{remark}

After constructing in the next section concrete circle action on the complex flag manifolds $F(m_1,\ldots,m_{r+1})$ with isolated fixed points and with the desired weights we shall see that Theorem \ref{mainresult} and Proposition \ref{mainprop} follow from Theorems \ref{BRF0} and \ref{BRF2} respectively.

\section{Complex flag manifolds and their coordinate charts}\label{preliminarysection}
In this preliminary section we recall the matrix-description and a typical local coordinate chart of complex flag manifolds , which are crucial to explicitly determine the weights of circle actions constructed in Section \ref{constructionsection}. Although these materials must be well-known to experts, we are unfortunately not able to find a suitable reference to them,
except a book written in Chinese by Q.-K. Lu (\cite{Lu}).
Therefore for reader's convenience we present these materials in detail in this section. As the description of the local coordinate charts of $F(m_1,\ldots,m_{r+1})$ is a little bit complicated for general $r$ and $m_i$, at least at the first glance, to make this section more accessible to readers not familiar with these materials, we would like to first illustrate the idea for complex Grassmannian in detail, which can be found, for example, in \cite[p. 193]{GH}.

First recall the following standard fact for complex projective spaces.
\begin{example}\label{examplecp}
\be
\mathbb{C}P^n:=\big\{(z_1,\ldots,z_{n+1})
\in\mathbb{C}^{n+1}\ \{0\}\big\}/\sim=:\{[z_1,\ldots,z_{n+1}]\},
\nonumber\ee
where
$(z_1,\ldots,z_{n+1})\sim(w_1,\ldots,w_{n+1})$ if and only if there exists $t\in\mathbb{C}-\{0\}$ such that $(z_1,\ldots,z_{n+1})=t(w_1,\ldots,w_{n+1})$
 and $[\cdot]$ denotes the coset element in the quotient space. Let
$$U_i:=\{[z_1,\ldots,z_{n+1}]\in\mathbb{C}P^n~|~z_i\neq 0\},\qquad 1\leq i\leq n+1.$$

Then
\be\label{cp}
\begin{split}
\varphi_i:~U_i&\stackrel{\cong}{\longrightarrow}\mathbb{C}^n\\ [z_1,\ldots,z_{n+1}]&\longmapsto(\frac{z_1}{z_i},\ldots,\widehat{\frac{z_i}{z_i}},\ldots,
\frac{z_{n+1}}{z_i}),\end{split}\ee
where $``\widehat{\cdot}"$ means deletion. These $(U_i,\varphi_i)$ $(1\leq i\leq n+1)$ form local coordinate charts for $\mathbb{C}P^n$.
\end{example}

The above construction in Example \ref{examplecp} can be extended to complex Grassmannian manifolds as follows (\cite[p. 193]{GH}).
\begin{example}\label{examplegrass}
Let $\text{Gr}(m,n)$ denotes the complex Grassmannian manifold consisting of complex $m$-dimensional sublinear spaces in $\mathbb{C}^{m+n}$. Note that $\text{dim}_{\mathbb{C}}\text{Gr}(m,n)=mn$. Denote by $M(r,s)$ the set of matrices with $r$ rows and $s$ columns and $\text{GL}(r,\mathbb{C})$ the general linear group of rank $r$ over $\mathbb{C}$. An element in $\text{Gr}(m,n)$ can be represented by a set of $m$ linearly independent column vectors in $\mathbb{C}^{m+n}$ spanning this element, i.e., by a matrix $A=(\alpha_1,\ldots,\alpha_m)\in M(m+n,m)$ of rank $m$ such that the column vectors $\alpha_1,\ldots,\alpha_m$ form a basis of this element. Obviously two matrices $A,B\in M(m+n,m)$ with $\text{rank}(A)=\text{rank}(B)=m$ represent the same element in $\text{Gr}(m,n)$ if and only if there exists $Q\in\text{GL}(m,\mathbb{C})$ such that $A=BQ$. Thus we have the following matrix-description for complex Grassmannian manifolds.
$$\text{Gr}(m,n)=\{[A]\}:=\{A~|~A\in M(m+n,m),\text{rank}(A)=m\}/\sim,$$
where $A\sim B$ if and only if there exists $Q\in\text{GL}(m,\mathbb{C})$ such that $A=BQ$.

With the notation given above, the local coordinate charts of $\text{Gr}(m,n)$ can be described in the following manner. Let $I=(i_1,\ldots,i_m)$ be an increasing integer sequence such that $1\leq i_1< i_2<\ldots< i_m\leq m+n$. For $A\in M(m+n,m)$, denote by $A_I\in M(m,m)$ the $I$-minor of $A$ consisting of its $i_1,\ldots,i_m$ rows, i.e., if $A:=(\beta_1,\ldots,\beta_{m+n})^{\text{t}}$, then $A_I:=(\beta_{i_1},\ldots,\beta_{i_{m}})^{\text{t}}$. Here ``$t$" denotes the transpose of a matrix. Then $$U_I:=\{[A]\in\text{Gr}(m,n), A_I\in\text{GL}(m,\mathbb{C})\}.$$
Clearly this definition is independent of the choice of $A$ in the coset as $A=BQ$ implies that $A_I=B_IQ$ and thus $U_I$ is well-defined. Note that inside each coset
$[A]\in U_I$ there contains a unique matrix representative such that its $I$-minor is the identity matrix. Indeed, arbitarily choose $A\in[A]$, $A\cdot(A_I)^{-1}$ is the desired representative. In this case, the resulting $mn$ entries in the matrix $A\cdot(A_I)^{-1}$ can be used to be the local coordinates of $[A]$. More precisely, if we define $I^c$ to be the increasing integer sequence complemenatry to $I$ with respect to $\{1,\ldots,m+n\}$, i.e., $I^c=\{1,\ldots,m+n\}-I,$ then
\be\label{gra}
\begin{split}\varphi_I:~U_I&\stackrel{\cong}{\longrightarrow}M(n,m)\cong\mathbb{C}^{mn},\\
[A]&\longmapsto A_{I^c}\cdot(A_I)^{-1}.\end{split}\ee
Note that the roles played by $A_{I^c}$ and $A_I$ in (\ref{gra}) are the same as those of $(z_1,\ldots,\widehat{z_i},\ldots,z_{n+1})$ and $z_i$ in (\ref{cp}).
\end{example}

With the construction of the local coordinate charts for complex Grassmannians in Example \ref{examplegrass} in mind, we can now proceed to the general complex flag manifolds $F(m_1,\ldots,m_r,m_{r+1})$.

We still use the notation and symbols introduced in Section \ref{section1}. A flag $(L_1,\ldots,L_r)\in F(m_1,\ldots,m_r,m_{r+1})$ can be represented by a matrix $A=(A_1,A_2,\ldots,A_r)\in M(N,\sum_{j=1}^rm_j)$ with $A_j\in M(N,m_j)$ and $\text{rank}(A)=\sum_{j=1}^rm_j$ such that the column vectors of the matrix $(A_1,\ldots,A_i)$ form a basis of the linear subspace $L_i$. The following lemma tells us that under what conditions two matrices represent the same flag.
\begin{lemma}\label{equivalentlemma}
Two matrices $$A=(A_1,A_2,\ldots,A_r), B=(B_1,B_2,\ldots,B_r)\in M(N,\sum_{j=1}^rm_j)$$
 with $\text{rank}(A)=\text{rank}(B)=\sum_{j=1}^rm_j$ represent the same flag $(L_1,\ldots,L_r)$ if and only if there exists a block upper triangular matrix $Q\in\text{GL}(\sum_{j=1}^rm_j,\mathbb{C})$ of the following form
 \be\label{lemmaQ}Q=\left(\begin{array}{cccc}
                             Q_{11} & Q_{12} &\cdots &Q_{1r}\\
                             0 & Q_{22} &\cdots&Q_{2r}\\
                             \vdots&\vdots&\ddots&\vdots\\
                             0&0&\cdots&Q_{rr}
                           \end{array}
                         \right),\qquad Q_{ii}\in\text{GL}(m_i,\mathbb{C}),\qquad
                         Q_{ij}\in M(m_i,m_j),\ee
                         such that $AQ=B$.
\end{lemma}
\begin{proof}
The ``only if" part. Suppose the matrices $A$ and $B$ represent the same flag $(L_1,\ldots,L_r)$, which implies that there exist $Q_{ji}, \widetilde{Q}_{ji}\in M(m_j,m_i)$ for $1\leq j\leq i\leq r$ such that
\be\label{lemmafor1}B_i=\sum_{j=1}^iA_jQ_{ji},\qquad A_i=\sum_{j=1}^iB_j\widetilde{Q}_{ji}.\ee
It suffices to show that these $Q_{ii}$ are nonsingular.
Indeed, we have from (\ref{lemmafor1}) that
\be\label{lemmafor2}B_i=\sum_{j=1}^i
\big[(\sum_{k=1}^jB_k\widetilde{Q}_{kj})Q_{ji}\big]=B_i\cdot
(\widetilde{Q}_{ii}Q_{ii})+\sum_{j=1}^{i-1}\big[B_j
(\sum_{k=j}^i\widetilde{Q}_{jk}Q_{ki})\big],\ee
Note that $\{B_1,\ldots,B_i\}$ is a basis of $L_i$. Then (\ref{lemmafor2}) tells us that
$$\widetilde{Q}_{ii}Q_{ii}=E_{m_i},\qquad \sum_{k=j}^i\widetilde{Q}_{jk}Q_{ki}=(0)\in M(m_j,m_i),$$
where $E_{m_i}$ is the identity matrix of rank $m_i$. This means that these $Q_{ii}$ are nonsingular.

The ``if" part. Suppose $A$ and $B$ represent two flags $(L_1,\ldots,L_r)$ and $(\widetilde{L}_1,\ldots,\widetilde{L}_r)$ respectively and satisfy (\ref{lemmaQ}). Then $B_i=\sum_{j=1}^iA_jQ_{ji}$. Then means that $\widetilde{L}_i\subset L_i$. Notice, however, that $$\text{dim}_{\mathbb{C}}\widetilde{L}_i=\text{dim}_{\mathbb{C}}L_i
=\sum_{j=1}^im_j,$$
which tells us that $\widetilde{L}_i=L_i.$
\end{proof}

With Lemma \ref{equivalentlemma} in hand we can now describe general complex flag manioflds as follows.
\be\begin{split} &F(m_1,\ldots,m_r,m_{r+1})=\{[A]\}\\
:=&\Big\{A=(A_1,A_2,\ldots,A_r)~\big|~A_j\in M(N,m_j),\text{rank}(A)
=\sum_{j=1}^rm_j\Big\}\big/\sim,\end{split}\ee
where $A\sim B$ if and only if $B=AQ$ with $Q$ satisfying (\ref{lemmaQ}).

Our next task is to give a description of local coordinate charts in the spirit of (\ref{gra}).

\begin{notation}
Suppose that $I=(I_1,\ldots,I_{r+1})$ is a decomposition of $\{1,2,\ldots,N\}$ (see (\ref{decomposition}) for its definition) and $A=(A_1,\ldots,A_r)$ with $A_i\in  M(N,m_i)$.
$$A_i^{(j)}:=\text{$I_j$-submatrix of $A_i$ consisting of the rows in $I_j$},$$
i.e., if $A_i=(\beta_1,\ldots,\beta_N)^{t}$ and $I_j=(k_1,\ldots,k_{m_j})$ with $1\leq k_1<\cdots<k_{m_j}\leq N$, then $$A_i^{(j)}:=(\beta_{k_1},\ldots,\beta_{k_{m_j}})^{t}\in M(m_j,m_i).$$
Define
$$A(I):=\left(\begin{array}{cccc}
                             A_1^{(1)} & A_2^{(1)} &\cdots &A_r^{(1)}\\
                             A_1^{(2)} & A_2^{(2)} &\cdots&A_r^{(2)}\\
                             \vdots&\vdots&\ddots&\vdots\\
                             A_1^{(r+1)}&A_2^{(r+1)}&\cdots&A_r^{(r+1)}
                           \end{array}
                         \right).$$

Note that $A(I)$ is nothing but rearranging the rows of $A$ in terms of the data $I=(I_1,\ldots,I_{r+1})$.
\end{notation}

For a decomposition $I=(I_1,\ldots,I_{r+1})$ of $\{1,\ldots,N\}$, we define
\be\label{UI}\begin{split}
 &U_I:=\\
 &\Big\{[(A_1,\ldots,A_r)]\in F(m_1,\ldots,m_{r+1})~\big|~
\begin{pmatrix}A_1^{(1)}&\cdots& A_i^{(1)}\\
\vdots&\ddots&\vdots\\
A_1^{(i)}&\cdots&A_i^{(i)}\end{pmatrix}~\text{are nonsingular for $1\leq i\leq r$}
\Big\}.\end{split}\ee
This definition is independent of the choice of the matrix representative in the coset as $(AQ)(I)=A_I\cdot Q$.
Our next proposition shows that these $U_I$ can be viewed as local coordinate charts.

\begin{proposition}\label{prop}~
\begin{enumerate}
\item
\be\label{unique1}\bigcup_IU_I=F(m_1,\ldots,m_{r+1}).\ee

\item
For each matrix representative $A=(A_1,\ldots,A_r)$ in $[A=(A_1,\ldots,A_r)]\in U_I$, there exists a unique $Q_A$ of the form as that of (\ref{lemmaQ}) such that $(A\cdot Q_A)(I)=A(I)\cdot Q_A$ is of the following form
\be\label{uniqueform}(A\cdot Q_A)(I)=A(I)\cdot Q_A
=\begin{pmatrix}E_{m_1}&0&\cdots&0\\
\ast&E_{m_2}&\cdots&0\\
\ast&\ast&\cdots&0\\
\vdots&\vdots&\ddots&\vdots\\
\ast&\ast&\cdots&E_{m_r}\\
\ast&\ast&\cdots&\ast
\end{pmatrix},\ee
where $E_m$ denotes the identity matrix of rank $m$.
\end{enumerate}
\end{proposition}
\begin{proof}~
\begin{enumerate}
\item
Suppose $A_i\in M(N,m_i)$ for $1\leq i\leq r$ and $\text{rank}(A_1,\ldots,A_r)=\sum_{i=1}^rm_i$. Then we can choose $m_1$ rows, say $1\leq k_1^{(1)}<\cdots<k_{m_1}^{(1)}\leq N$, such that these rows of $A_1$ are linearly independent as $\text{rank}(A_1)=m_1$. We denote by $I_1:=(k_1^{(1)},\ldots,k_{m_1}^{(1)})$. By our choice the $k_1^{(1)},\ldots,k_{m_1}^{(1)}$ rows of the matrix $(A_1,A_2)$ are also linearly independent. Since $\text{rank}(A_1,A_2)=m_1+m_2$, we are able to supplement these $m_1$ rows with another $m_2$ rows, say $k_1^{(2)}<\cdots<k_{m_2}^{(2)}$, such that these $m_1+m_2$ rows of $(A_1,A_2)$ are linearly independent. We denote by $I_2:=(k_1^{(2)},\ldots,k_{m_2}^{(2)})$. We continue to apply this idea to obtain $I_i:=(k_1^{(i)},\ldots,k_{m_i}^{(i)})$ ($i\leq r$) such that the rows whose indices are contained in $I_1,\ldots,I_i$ of the matrix $(A_1,\ldots,A_i)$ are linearly independent. Denote by $I_{r+1}:=\{1,\ldots,N\}-\bigcup_{1\leq i\leq r}I_i$ and $I:=(I_1,\ldots,I_{r+1})$. Then $[(A_1,\ldots,A_r)]\in U_I$ for this chosen $I$. This completes the proof of (\ref{unique1}).

\item
Note that the submatrices $Q_{1i},\ldots,Q_{ii}$ in $Q_A$ are characterized by the following equations
$$\begin{pmatrix}A_1^{(1)}&\cdots& A_i^{(1)}\\
\vdots&\ddots&\vdots\\
A_1^{(i)}&\cdots&A_i^{(i)}\end{pmatrix}
\begin{pmatrix}Q_{1i}\\Q_{2i}\\\vdots\\Q_{ii}\end{pmatrix}
=
\begin{pmatrix}0\\\vdots\\0\\E_{m_i}\end{pmatrix},\qquad 1\leq i\leq r,$$
whose existence and uniqueness are then yielded by the invertibility of the matrices $(A_{r}^{(s)})_{1\leq r,s\leq i}$.
\end{enumerate}
\end{proof}

Proposition \ref{prop} tells us that each coset
$[A=(A_1,\ldots,A_r)]\in U_I$ contains a unique matrix representative such that it is of the form of the right hand side of (\ref{uniqueform}) after rearranging its rows in terms of $I$. Now we can use the bottom left entries of this unique matrix representative to be the local coordinates. To be more precise, we have
\be\label{flag}
\begin{split}\varphi_I:~U_I&\stackrel{\cong}{\longrightarrow}
\mathbb{C}^{d},\qquad(d=\sum_{1\leq i<j\leq r+1}m_im_j)\\
[A]&\longmapsto \text{the bottom left entries in $(A\cdot Q_A)(I)$}.\end{split}\ee
This, together with (\ref{unique1}), implies that $\{(U_I,\varphi_I)~|~I\}$ gives the desired local coordinate charts for the complex flag manifold $F(m_1,\ldots,m_{r+1})$.

\section{Circle actions on complex flag manifolds and proof of main results}\label{constructionsection}
In this section we construct a holomorphic circle action on the complex flag manifolds $F(m_1,\ldots,m_r,m_{r+1})$, show that it has isolated fixed points, and explicitly determine the weights on these fixed points. Then Theorems \ref{BRF0} and \ref{BRF2} will yield our main results in Section \ref{mainresultsection}.

We arbitrarily choose $N=\sum_{j=1}^{r+1}m_j$ mutually distinct integers $k_1,\ldots,k_N$ and use them to construct the following circle action.
\be\label{circleaction}\begin{split}
S^1\times F(m_1,\ldots,m_{r+1})&\stackrel{\psi}{\longrightarrow}F(m_1,\ldots,m_{r+1}),\\
\big(g,[A=(A_1,\ldots,A_r)]\big)&\longmapsto
\Big[\left(\begin{array}{ccc}
g^{k_1}&  &\\
& \ddots &\\
& &g^{k_N}
\end{array}
\right)A\Big],
\end{split}\ee
i.e., the action of $g\in S^1$ is given by multiplying the entries of the $i$-row of the matrix $A$ with $g^{k_i}$. It is clear that $\psi$ is well-defined and gives rise to a circle action on $F(m_1,\ldots,m_{r+1})$. Define $$\psi_g([A]):=\psi(g,[A]).$$
By the definition of $U_I$ in (\ref{UI}) we know that $\psi_g(U_I)=U_I$ for any decomposition $I$. Recall in (\ref{flag}) the local coordinate description of $U_I$ under $\varphi_I$ and let
$$\tilde{\psi}_g: \mathbb{C}^d\stackrel{\cong}{\longrightarrow}
\mathbb{C}^d\qquad(d=\sum_{1\leq
i<j\leq r+1}m_im_j)$$
be the map on the level of Euclidean space induced from $\psi_g$ and $\varphi_I$, i.e., we have the following commutative diagram:

\be\label{psi}\xymatrix{
U_I \ar[d]^{\cong}_{\varphi_I} \ar[r]_{\cong}^{\psi_g} &U_I\ar[d]^{\cong}_{\varphi_I}\\
\mathbb{C}^d \ar[r]_{\cong}^{\tilde{\psi}_g} &\mathbb{C}^d}.\ee

For each decomposition $I=(I_1,\ldots,I_{r+1})$, if $I_i=(t_1,\ldots,t_{m_i})$ $(t_1<\cdots<t_{m_i})$, we simply denote by $(g^{\pm k_{I_i}})$ the rank $I_i$ diagonal matrix whose entries are $g^{\pm k_{t_1}},\ldots,g^{\pm k_{t_{m_i}}}$:
$$(g^{k_{I_i}}):=\left(\begin{array}{ccc}
g^{k_{t_1}}&  &\\
& \ddots &\\
& &g^{k_{t_{m_i}}}
\end{array}
\right),\qquad (g^{-k_{I_i}}):=\left(\begin{array}{ccc}
g^{-k_{t_1}}&  &\\
& \ddots &\\
& &g^{-k_{t_{m_i}}}
\end{array}
\right).$$

With this convention understood, we now prove the following key lemma, which describes the behavior of the map $\tilde{\psi}_g$ and contains all key information in establishing our main results.
\begin{lemma}\label{keylemma}
The map $\tilde{\psi}_g$ in {\rm(}\ref{psi}{\rm)} behaves as follows:
\be\label{mappsi}\begin{split}
&\tilde{\psi}_g\Big(\begin{pmatrix}E_{m_1}&0&\cdots&0\\
B_{21}&E_{m_2}&\cdots&0\\
B_{31}&B_{32}&\cdots&0\\
\vdots&\vdots&\ddots&\vdots\\
B_{r1}&B_{r2}&\cdots&E_{m_r}\\
B_{r+1,1}&B_{r+1,2}&\cdots&B_{r+1,r}
\end{pmatrix}\Big)\\
=&
\begin{pmatrix}E_{m_1}&0&\cdots&0\\
(g^{k_{I_2}})B_{21}(g^{-k_{I_1}})&E_{m_2}&\cdots&0\\
(g^{k_{I_3}})B_{31}(g^{-k_{I_1}})&(g^{k_{I_3}})B_{32}
(g^{-k_{I_2}})&\cdots&0\\
\vdots&\vdots&\ddots&\vdots\\
(g^{k_{I_r}})B_{r1}(g^{-k_{I_1}})&
(g^{k_{I_r}})B_{r2}(g^{-k_{I_2}})&\cdots&E_{m_r}\\
(g^{k_{I_{r+1}}})B_{r+1,1}(g^{-k_{I_1}})&
(g^{k_{I_{r+1}}})B_{r+1,2}(g^{-k_{I_2}})&\cdots&
(g^{k_{I_{r+1}}})B_{r+1,r}(g^{-k_{I_r}})
\end{pmatrix}.
\end{split}\ee
Here, as we have done in {\rm(}\ref{flag}{\rm)}, still denote by the bottom left entries of the matrix the coordinate components of $\mathbb{C}^d$.
{\rm(}\ref{mappsi}{\r)} particularly implies that the action $\psi_g$ is holomorphic for any $g\in S^1$. This means that the circle action $\psi$ constructed in {\rm(}\ref{circleaction}{\rm)} is holomorphic.
\end{lemma}
\begin{proof}
Recall the definitions of $\varphi_I$ \big((\ref{uniqueform}) and (\ref{flag})\big) and $\psi_g$. It suffices to show that, for $[A]\in U_I$, if $A(I)\cdot Q_A$ is equal to the matrix on the left hand side of (\ref{mappsi}), then $$\big(\psi_g(A)\big)(I)\cdot Q_{\psi_g(A)}$$ is equal to the matrix on the right hand side of (\ref{mappsi}).

First note that
\be\big(\psi_g(A)\big)(I)=\Big(\left(\begin{array}{ccc}
g^{k_1}&  &\\
& \ddots &\\
& &g^{k_N}
\end{array}
\right)A\Big)(I)=
\left(\begin{array}{ccc}
(g^{k_{I_1}})&  &\\
& \ddots &\\
& &(g^{k_{I_{r+1}}})
\end{array}
\right)\cdot A(I).
\nonumber\ee

Therefore,
\be\begin{split}
&\big(\psi_g(A)\big)(I)\cdot Q_A\\
=&\left(\begin{array}{ccc}
(g^{k_{I_1}})&  &\\
& \ddots &\\
& &(g^{k_{I_{r+1}}})
\end{array}
\right)\cdot A(I)\cdot Q_A\\
=&\left(\begin{array}{ccc}
(g^{k_{I_1}})&  &\\
& \ddots &\\
& &(g^{k_{I_{r+1}}})
\end{array}
\right)\begin{pmatrix}E_{m_1}&0&\cdots&0\\
B_{21}&E_{m_2}&\cdots&0\\
B_{31}&B_{32}&\cdots&0\\
\vdots&\vdots&\ddots&\vdots\\
B_{r1}&B_{r2}&\cdots&E_{m_r}\\
B_{r+1,1}&B_{r+1,2}&\cdots&B_{r+1,r}
\end{pmatrix}\\
=&\begin{pmatrix}E_{m_1}&0&\cdots&0\\
(g^{k_{I_2}})B_{21}(g^{-k_{I_1}})&E_{m_2}&\cdots&0\\
(g^{k_{I_3}})B_{31}(g^{-k_{I_1}})&(g^{k_{I_3}})B_{32}
(g^{-k_{I_2}})&\cdots&0\\
\vdots&\vdots&\ddots&\vdots\\
(g^{k_{I_r}})B_{r1}(g^{-k_{I_1}})&
(g^{k_{I_r}})B_{r2}(g^{-k_{I_2}})&\cdots&E_{m_r}\\
(g^{k_{I_{r+1}}})B_{r+1,1}(g^{-k_{I_1}})&
(g^{k_{I_{r+1}}})B_{r+1,2}(g^{-k_{I_2}})&\cdots&
(g^{k_{I_{r+1}}})B_{r+1,r}(g^{-k_{I_r}})
\end{pmatrix}
\left(\begin{array}{ccc}
(g^{k_{I_1}})&  &\\
& \ddots &\\
& &(g^{k_{I_{r}}})
\end{array}
\right).
\end{split}
\nonumber\ee

Thus we have established
\be\label{formula}
\big(\psi_g(A)\big)(I)\cdot \Big[Q_A\left(\begin{array}{ccc}
(g^{-k_{I_1}})&  &\\
& \ddots &\\
& &(g^{-k_{I_{r}}})
\end{array}
\right)\Big]=\text{the matrix on the right hand side of (\ref{mappsi})}.
\ee

Note that $Q_A$ is block upper triangular and so is the product matrix inside $[\cdot]$ on the left hand side of (\ref{formula}). The uniqueness of $Q_{\psi_g(A)}$ showed in \big(Prop. \ref{prop}, (2)\big) tells us that this product matrix is precisely $Q_{\psi_g(A)}$ and thus yields the desired proof.
\end{proof}

With this key lemma in hand, we are now ready to show the following results and complete the proof of our main results in Section \ref{mainresultsection}.
\begin{proposition}~
\begin{enumerate}
\item
The fixed points of the holomorphic circle action are indexed by the decompositions $I$ of $\{1,\ldots,N\}$, say $\{P_I\}$. More precisely, $P_I=\varphi_I^{-1}(0)$, where $0$ denotes the origin of $\mathbb{C}^d$.

\item
The weights around $P_I$ induced by the circle action on the holomorphic tangent space to $P_I$ are
\be\label{weight}\coprod_{1\leq i<j\leq r+1}\big\{-k_{\alpha}+k_{\beta}~|~\alpha\in I_i,~\beta\in I_j\big\},\qquad I=(I_1,\ldots,I_{r+1}).\ee
Here ``$\coprod$'' means disjoint union, i.e., possibly repeated integers cannot be discarded.

\item
Theorem \ref{mainresult} and Proposition \ref{mainprop} hold.
\end{enumerate}
\end{proposition}

\begin{proof}
\begin{enumerate}
\item
Suppose $P$ is some fixed point of the circle action $\psi$. This means $P$ is fixed by $\psi_g$ for each $g$. Assume by (\ref{unique1}) that $P\in U_I$ for some $I$. Then $\varphi_I(P)$ is fixed by $\tilde{\psi}_g$ for each $g$. This implies that, the coordinate functions in (\ref{mappsi}) for $\varphi_I(P)$ satisfy
\be\label{anyg}B_{ji}=(g^{k_{I_j}})B_{ji}(g^{k_{-I_i}}) ~(1\leq i<j\leq r+1),\qquad\forall~g\in S^1.\ee
Clearly the unique solution to (\ref{anyg}) is all these matrices $B_{ji}=(0)$ and thus the unique fixed point in $U_I$ is $\varphi_I^{-1}(0)=:P_I$.

\item
Since the tangent space to the origin of $\mathbb{C}^d$ can be canonically identified with $\mathbb{C}^d$ itself. Thus the tangent map of $\psi_g$ at $P_I$ can be identified with $\tilde{\psi}_g$. However, (\ref{mappsi}) tells us that $\tilde{\psi}_g$ sends $B_{ji}$ to $(g^{k_{I_j}})B_{ji}(g^{k_{-I_i}})$, i.e., this map is given by multiplying the entries of $B_{ji}$, which are viewed as the coordinate components of $\mathbb{C}^d$, with $g^{-k_{\alpha}+k_{\beta}}$ ($\alpha\in I_i,\beta\in I_j$), which, according to the definition of weight \big(cf. (\ref{weightdefinitin})\big), leads to (\ref{weight}).

\item
Knowing the concrete values of the weights (\ref{weight}) around the isolated fixed points of the holomorphic circle action $\psi$, we now directly apply Theorems \ref{BRF0} and \ref{BRF2} to this model to yield the conclusions in Theorem \ref{mainresult} and Proposition \ref{mainprop}, with the only difference that the indeterminates $x_i$ be replaced by the integers $k_i$. Note that the choice of these mutually distinct integers $k_i$ is completely arbitrary. This means that these equalities hold as identities, i.e., we have the desired conclusions as stated in Theorem \ref{mainresult} and Proposition \ref{mainprop}.
\end{enumerate}
\end{proof}

\section{Appendix}\label{appendix}
In this appendix we apply Theorem \ref{mainresult} to work out (\ref{chernnumber}) again. Since the calculations for $c_1^5[F(1,1,2)]$ and $c_1^5[F(1,2,1)]$ are identically the same. We only demonstrate the former in detail.

For $F(1,1,2)$, we have $r=2,$ $m_1=m_2=1$, $m_3=2$ and $N=4$. There are $12$ decompositions of the set $\{1,2,3,4\}:$
$$I_{(1,2)}:=(\{1\},\{2\},\{3,4\}),\qquad I_{(1,3)}:=(\{1\},\{3\},\{2,4\}),\qquad
I_{1,4}:=(\{1\},\{4\},\{2,3\}),$$
$$I_{(2,1)}:=(\{2\},\{1\},\{3,4\}),\qquad I_{(2,3)}:=(\{2\},\{3\},\{1,4\}),\qquad
I_{2,4}:=(\{2\},\{4\},\{1,3\}),$$
$$I_{(3,1)}:=(\{3\},\{1\},\{2,4\}),\qquad I_{(3,2)}:=(\{3\},\{2\},\{1,4\}),\qquad
I_{3,4}:=(\{3\},\{4\},\{1,2\}),$$
$$I_{(4,1)}:=(\{4\},\{1\},\{2,3\}),\qquad I_{(4,2)}:=(\{4\},\{2\},\{1,3\}),\qquad
I_{4,3}:=(\{4\},\{3\},\{1,2\}).$$
As we have remarked in Remark \ref{remarkaftertheorem}, we may assume that $x_i=i$ $(1\leq i\leq4)$ to calculate the Chern number.
For simplicity, we denote by
$$W_{(i,j)}:=W_{I_{(i,j)}},\qquad e_{(i,j)}:=e(W_{(i,j)}),\qquad c_{1(i,j)}:=c_1(W_{(i,j)}).$$
 Then we have
 \begin{eqnarray}
\left\{ \begin{array}{ll}
\big(W_{(1,2)},e_{(1,2)},c_{1(1,2)}\big)
=\big(\{1,2,3,1,2\},12,9\big),\\
\big(W_{(1,3)},e_{(1,3)},c_{1(1,3)}\big)
=\big(\{2,1,3,-1,1\},-6,6\big),\\
\big(W_{(1,4)},e_{(1,4)},c_{1(1,4)}\big)
=\big(\{3,1,2,-2,-1\},12,3\big),\\
\big(W_{(2,1)},e_{(2,1)},c_{1(2,1)}\big)
=\big(\{-1,1,2,2,3\},-12,7\big),\\
\big(W_{(2,3)},e_{(2,3)},c_{1(2,3)}\big)
=\big(\{1,-1,2,-2,1\},4,1\big),\\
\big(W_{(2,4)},e_{(2,4)},c_{1(2,4)}\big)
=\big(\{2,-1,1,-3,-1\},-6,-2\big),\\
\big(W_{(3,1)},e_{(3,1)},c_{1(3,1)}\big)
=\big(\{-2,-1,1,1,3\},6,2\big),\\
\big(W_{(3,2)},e_{(3,2)},c_{1(3,2)}\big)
=\big(\{-1,-2,1,-1,2\},-4,-1\big),\\
\big(W_{(3,4)},e_{(3,4)},c_{1(3,4)}\big)
=\big(\{1,-2,-1,-3,-2\},12,-7\big),\\
\big(W_{(4,1)},e_{(4,1)},c_{1(4,1)}\big)
=\big(\{-3,-2,-1,1,2\},12,3\big),\\
\big(W_{(4,2)},e_{(4,2)},c_{1(4,2)}\big)
=\big(\{-2,-3,-1,-1,1\},6,-6\big),\\
\big(W_{(4,3)},e_{(4,3)},c_{1(4,3)}\big)
=\big(\{-1,-3,-2,-2,-1\},-12,9\big).
\end{array} \right.\nonumber
\end{eqnarray}
Therefore, Theorem \ref{mainresult} tells us that
$$c_1^5[F(1,1,2)]=\sum_{1\leq i\neq j\leq 4}\frac{c_{1(i,j)}^5}{e_{(i,j)}}=\frac{12^5}{9}+\cdots+\frac{(-12)^5}{-9}=4500.$$
We also note that in this case
$$\sum_{1\leq i\neq j\leq 4}\frac{c_{1(i,j)}^6}{e_{(i,j)}}=\frac{12^6}{9}+\cdots+\frac{(-12)^6}{-9}=0,$$
which is consistent with the first equality in (\ref{d}).

\end{document}